\documentclass[11pt]{article}

\usepackage[margin=0.7in]{geometry}
\usepackage[T1]{fontenc}
\usepackage[utf8]{inputenc}
\usepackage{lmodern}
\usepackage{microtype}
\usepackage{amsmath,amssymb}
\usepackage{graphicx}
\usepackage{subcaption}
\usepackage{xcolor}
\usepackage[colorlinks=true,allcolors=blue]{hyperref}
\usepackage[round,authoryear]{natbib}

\usepackage{amsthm}
\newtheorem{theorem}{Theorem}
\newtheorem{proposition}{Proposition}
\newtheorem{lemma}{Lemma}

\title{Climate-model factor separation with Shapley values and efficient sampling}
\author{Abel Jansma}
\date{Dutch Institute for Emergent Phenomena\\
Institute of Physics \& Institute for Logic, Language and Computation\\
University of Amsterdam\\\today}

\begin{document}

\maketitle

\begin{abstract}
Climate models often feature nonlinear responses to changes in model parameters or boundary conditions. Factor separation asks how much of the resulting change should be assigned to altered factors and their interactions. \citet{lunt2021multi} showed that two factor attribution methods, the \emph{linear-sum} and \emph{shared-interaction} methods, coincide for $n\le 4$ factors and conjectured this holds for arbitrary $n$. We show that the linear-sum and shared-interaction formulas are respectively the permutation and Harsanyi-dividend representations of the Shapley value. Their equality therefore follows from a classical theorem in cooperative game theory. To our knowledge, this is the first explicit identification of the Stein--Alpert and Lunt climate-model factor-separation methods with Möbius/Harsanyi coefficients and Shapley values. The result imports many established and efficient Shapley sampling methods into climate-model experimental design, and extends to vector-valued responses. 
\end{abstract}

\section{Introduction}

Climate-model experiments and simulations often compare a baseline state with a perturbed state in which several conditions or model parameters have changed. The method of factor separation, introduced by \citet{stein1993factor}, asks how much of the resulting change in prediction should be attributed to each altered factor and to interactions among factors. Examples include separating the effects on temperature of atmospheric $\mathrm{CO}_2$, ice sheets, orography, vegetation, land cover, and emissions.

When the system's response is nonlinear, the effect assigned to one factor depends on which other factors have already been changed. Studying all possible combinations of $n$ factors thus requires $2^n$ simulations, introduces $2^n$ model responses, and defines an $n$-dimensional cube of model states. This makes the attribution problem combinatorially complex and difficult to interpret.

\citet{lunt2021multi} proposed three extended methods for multivariate factorisation of numerical simulations: linear-sum, shared-interaction, and scaled-residual. They distinguished four desirable properties of an attribution: completeness, meaning that the factor contributions sum to the total change; purity, meaning that no separate residual term is required; uniqueness, meaning that the result does not depend on an arbitrarily selected path; and a reversal property that they call symmetry, meaning that reversing the endpoint comparison reverses the signs of the attributed effects. The present paper identifies only the linear-sum and shared-interaction methods with Shapley values. The linear-sum method averages over all monotone paths through the cube of model states. The shared-interaction method calculates the Stein--Alpert interaction terms and divides each term equally among the factors that generate it. Lunt et al.\ showed that these two methods coincide for up to four factors, with equality for arbitrary $n$ stated as a conjecture.

We prove that conjecture. The key is to distinguish two related tasks: (i) \emph{decomposition}, meaning the identification of all $2^n$ factor interaction terms; and (ii) \emph{attribution}, meaning the assignment of $n$ attribution values to the individual factors.

The decomposition task is solved by Möbius inversion on the lattice of factor subsets, which rewrites the Stein--Alpert interactions. The attribution task is solved by Shapley values. It is a classical result that Shapley values can equivalently be computed either by averaging marginal contributions over player orderings or by dividing each Möbius coefficient, or Harsanyi dividend, equally among the members of its coalition. Our contribution is to identify the linear-sum and shared-interaction factorisations of \citet{lunt2021multi} with these two representations, thereby settling their conjecture and importing established Shapley-value sampling and approximation methods into this part of the factor-separation literature.

\section{Related work and scope of the contribution}

\citet{cleveland2020factor} connect Stein--Alpert factor separation to the design-of-experiments literature. They relate factor effects in numerical simulations to simple effects, regression coding, reduced experimental designs, and factors with more than two levels. That work clarifies the statistical structure of factor-separation experiments, but it does not make the specific identification between the linear-sum/shared-interaction formulas and the Shapley value.

Shapley values do have prior applications in the climate and energy literature. \citet{albrecht2002shapley} use Shapley values to study effects on historical carbon emissions, and \citet{alkhourdajie2024climate} use a Shapley-based attribution method to analyse energy sector indicators in climate mitigation scenarios. In the field of sensitivity analysis, \citet{owen2014sobol} and \citet{song2016shapley} developed Shapley effects, and in the field of machine learning, Shapley values are widely used to interpret model predictions \citep{lundberg2017unified}. 

\section{Stein--Alpert factor separation as Boolean calculus}

Let $\mathcal{N}=\{1,\ldots,n\}$ be the set of binary factors, and let $F:2^{\mathcal{N}}\to V$ be the raw model response, where $V$ is a real vector space. The baseline-normalised response is
\begin{equation}
    v(S)=F(S)-F(\emptyset),
    \qquad S\subseteq\mathcal{N},
\end{equation}
so $v(\emptyset)=0$. For the scalar climate variables emphasised by \citet{stein1993factor} and \citet{lunt2021multi}, one may take $V=\mathbb{R}$.

\citet{stein1993factor} propose that the model response to the presence of a subset $S\subseteq\mathcal{N}$ of factors can be decomposed into irreducible contributions $\hat{v}(T)$ as follows:

\begin{align}
    v(S) = \sum_{T\subseteq S} \hat{v}(T). \label{eq:sum_of_subsets}
\end{align}
This rewrites their Taylor-series equation (10) as a function over sets, rather than indices, to cast it in the language of quantitative mereology \citep{jansma2025mereological}. To derive the interaction coefficients $\hat{v}(T)$, \citet{stein1993factor} recursively solve the associated system of equations. Equivalently, one can apply the Möbius inversion theorem to \eqref{eq:sum_of_subsets} and obtain a closed-form expression for the interaction coefficients \citep{rota1964foundations}:
\begin{align}
    \hat{v}(S) = \sum_{T\subseteq S} (-1)^{|S| - |T|} v(T),
    \label{eq:mobius}
\end{align}
which matches their Equation (16). This also makes the completeness property of the decomposition explicit, since the Möbius inversion theorem guarantees that the decomposition is unique, invertible, and satisfies \eqref{eq:sum_of_subsets} exactly. 

To go from decomposition to attribution, we turn to cooperative game theory, where the function $v$ is called a game's characteristic-, or value-function and the irreducible contributions $\hat{v}$ are called \emph{Harsanyi dividends} \citep{harsanyi1959bargaining}. The Shapley attribution to factor $i$ can be written in dividend form as
\begin{align}
    \phi_i(v) = \sum_{\{i\}\subseteq S\subseteq \mathcal{N}} \frac{\hat{v}(S)}{|S|}. \label{eq:shap_synergy}
\end{align}
This coincides with the \emph{shared-interaction} factorisation $\Delta T_i(v)$ introduced in Equation (13) of \citet{lunt2021multi}. In the scalar case, the Shapley value is characterised by efficiency, additivity or linearity, the null-player/dummy axiom, and player-exchange symmetry \citep{shapley1953value}. This player-exchange symmetry is distinct from the endpoint-reversal antisymmetry that \citet{lunt2021multi} call symmetry. To make the latter explicit, let
\begin{equation}
    v^{\mathrm{rev}}(S)=v(\mathcal{N}\setminus S)-v(\mathcal{N})
    \label{eq:reversed_game}
\end{equation}
denote the baseline-normalised response obtained by exchanging the baseline and fully perturbed endpoints. An attribution rule $A_i$ satisfies endpoint-reversal antisymmetry when
\begin{equation}
    A_i(v^{\mathrm{rev}})=-A_i(v)
    \qquad\text{for all }i.
    \label{eq:endpoint_reversal}
\end{equation}
Lunt et al.'s four properties are therefore not, by themselves, a uniqueness characterization of the Shapley value. Under the Shapley characterization, the shared-interaction attribution satisfies:
\begin{itemize}
    \item \emph{Efficiency}: $\sum_{i=1}^{n}\phi_i(v)=v(\mathcal{N})$.
    \item \emph{Linearity}: $\phi_i(av+bw)=a\phi_i(v)+b\phi_i(w)$ for responses $v,w$ and scalars $a,b$.
    \item \emph{Null player/factor}: if $v(S\cup\{i\})=v(S)$ for all $S\subseteq\mathcal{N}\setminus\{i\}$, then $\phi_i(v)=0$.
    \item \emph{Player-exchange symmetry}: if $v(S\cup\{i\})=v(S\cup\{j\})$ for all $S\subseteq\mathcal{N}\setminus\{i,j\}$, then $\phi_i(v)=\phi_j(v)$.
\end{itemize}

\section{Linear-sum factorisation as Shapley values}
\citet{lunt2021multi} also proposed the \emph{linear-sum} method (their Equation (8)). Following their construction: let $A_i$ denote the collection of all subsets of $\mathcal{N}$ containing $i$, and let $B_i$ denote the collection of all subsets of $\mathcal{N}$ not containing $i$, ordered such that their $j$-th elements differ only by the presence of $i$: $B_i^j\cup \{i\} = A_i^j$. Then the linear-sum attribution is defined as the following weighted average of marginal effects:
\begin{align}
    \Delta T_i(v) = \frac{1}{n!} \sum_{j=1}^{2^{n-1}} |B_i^j|!(n-1-|B_i^j|)!\left( v(A_i^j) - v(B_i^j) \right).
\end{align}
The sum over $j$ can be written as a sum over subsets not containing $i$, yielding the equivalent expression
\begin{align}
    \Delta T_i(v) = \sum_{S\subseteq \mathcal{N}\setminus \{i\}} \frac{|S|!(n-1-|S|)!}{n!} \left( v(S\cup \{i\}) - v(S) \right). \label{eq:shap_orig}
\end{align}
This is exactly the original definition of the Shapley value $\phi_i(v)$ \citep{shapley1953value}, and therefore coincides with the shared-interaction method. This proves the conjecture of \citet{lunt2021multi} as a corollary of the equivalence of \eqref{eq:shap_orig} and \eqref{eq:shap_synergy}: the linear-sum and shared-interaction factor attributions are equivalent for arbitrary $n$. In game theory, this equivalence is well established, but for completeness we give a proof below.

\begin{theorem}[Conjectured by \citet{lunt2021multi}]
\label{thm:main}
    Let $\mathcal{N}=\{1,\ldots,n\}$ be a set of $n$ binary factors, let $V$ be a real vector space, and let $v:2^{\mathcal{N}}\to V$ be the baseline-normalised response. Then the linear-sum attribution and the shared-interaction attribution of $v$ are equivalent for arbitrary $n$:
    \begin{align}
        \Delta T_i(v) = \sum_{S\subseteq \mathcal{N}\setminus \{i\}} \frac{|S|!(n-1-|S|)!}{n!} \left( v(S\cup \{i\}) - v(S) \right) = \sum_{\{i\}\subseteq S\subseteq \mathcal{N}} \frac{\hat{v}(S)}{|S|} = \phi_i(v),
    \end{align}
    where $\hat{v}(S) = \sum_{T\subseteq S} (-1)^{|S| - |T|} v(T)$ is the Möbius inversion of $v$ over the power set of $\mathcal{N}$.
\end{theorem}
\begin{proof}
    First note that the linear-sum attribution can be written as a sum over all permutations $\pi$ of $\mathcal{N}$, which is how \citet{lunt2021multi} originally motivated it. Given a permutation $\pi$, let $P_i^\pi$ denote the set of factors that precede $i$ in the ordering $\pi$. Then we can write
    \begin{align}
        \Delta T_i(v)&= \sum_{S\subseteq \mathcal{N}\setminus \{i\}} \frac{|S|!(n-1-|S|)!}{n!} \left( v(S\cup \{i\}) - v(S) \right)\\
        &= \frac{1}{n!} \sum_{\pi\in \mathrm{Perm}(\mathcal{N})} \left( v(P_i^\pi\cup \{i\}) - v(P_i^\pi) \right).
    \end{align}
    We can then apply the Möbius inversion theorem to write
    \begin{align}
        v(P_i^\pi\cup\{i\})-v(P_i^\pi)
        = \sum_{S:\{i\}\subseteq S\subseteq P_i^\pi\cup\{i\}} \hat{v}(S).
    \end{align}
    For a fixed $S$ that contains $i$, and a uniformly chosen permutation $\pi$, the probability that $S\subseteq P_i^\pi\cup\{i\}$ is exactly $1/|S|$, since all members of $S$ are equally likely to be the last member of $S$ in the ordering $\pi$. Therefore, we can write
    \begin{align}
        \Delta T_i(v) &= \frac{1}{n!} \sum_{\pi\in \mathrm{Perm}(\mathcal{N})} \sum_{S:\{i\}\subseteq S\subseteq P_i^\pi\cup\{i\}} \hat{v}(S)\\
        &= \sum_{\{i\}\subseteq S\subseteq \mathcal{N}} \frac{\hat{v}(S)}{|S|},
    \end{align}
    which is exactly the shared-interaction attribution $\phi_i(v)$.
\end{proof}

\section{Permutation sampling of Shapley factor attributions}
\label{sec:sampling}

The equivalence established in Theorem~\ref{thm:main} has an immediate computational consequence. Exact recovery of the complete Möbius decomposition generally requires the full set of $2^n$ model responses, whereas the permutation representation expresses each factor attribution as an expectation over random factor orderings. The latter can therefore be estimated without reconstructing the individual interaction dividends. More precisely, a particular Möbius coefficient $\hat v(S)$ is the alternating sum in~\eqref{eq:mobius} over the $2^{|S|}$ subsets of $S$, while the full shared-interaction representation~\eqref{eq:shap_synergy} uses all nonempty coefficients. The computational advantage of sampling is therefore an advantage for factor-level attribution, not for reconstructing the complete interaction decomposition. In climate applications, where each evaluation of $v$ may be a full simulation, exact calculations become prohibitive as the number or cost of simulations increases.

Write the marginal contribution of factor $i$ under a permutation $\pi$ as
\begin{equation}
  \Delta_i(\pi) \;=\; v\!\left(P_i^{\pi}\cup\{i\}\right) - v\!\left(P_i^{\pi}\right),
  \label{eq:marginal}
\end{equation}
where $P_i^{\pi}$ is the set of factors preceding $i$ in $\pi$, as in the proof of Theorem~\ref{thm:main}. Then $\Delta T_i(v)=\mathbb{E}_{\pi}\!\left[\Delta_i(\pi)\right]$ with $\pi$ uniform on $\mathrm{Perm}(\mathcal{N})$. Drawing $\pi_1,\dots,\pi_M$ independently and uniformly gives the following estimator, first used by \citet{mann1960values} and studied in more detail by \citet{castro2009polynomial}:
\begin{equation}
  \widehat{\Delta T}_i^{(M)}
    \;=\; \frac{1}{M}\sum_{m=1}^{M}\Delta_i(\pi_m),
  \label{eq:estimator}
\end{equation}
which is unbiased:
\begin{equation}
  \mathbb{E}\!\left[\widehat{\Delta T}_i^{(M)}\right]=\Delta T_i(v),
  \qquad
  \operatorname{Var}\!\left(\widehat{\Delta T}_i^{(M)}\right)
  =\frac{\sigma_i^2}{M},
  \qquad
  \sigma_i^2=\operatorname{Var}_{\pi}[\Delta_i(\pi)].
\end{equation}
The empirical variance of the sampled marginal contributions is
\begin{equation}
  s_i^2
  =
  \frac{1}{M-1}
  \sum_{m=1}^{M}
  \left(
  \Delta_i(\pi_m)-\widehat{\Delta T}_i^{(M)}
  \right)^2.
\end{equation}
Under independent permutation sampling, $s_i/\sqrt{M}$ is an estimated Monte Carlo standard error, and the scalar central limit theorem gives the pointwise asymptotic confidence interval (though non-asymptotic bounds are also available \citep{maleki2013bounding}):
\begin{equation}
  \widehat{\Delta T}_i^{(M)}
  \pm
  z_{1-\alpha/2}\frac{s_i}{\sqrt{M}}.
\end{equation}

A single permutation should be reused for all $n$ factors: the ordering
$\pi$ traces a maximal chain
\[
\varnothing = S_0 \subset S_1 \subset \cdots \subset S_n = \mathcal{N}.
\]
The successive differences $v(S_k)-v(S_{k-1})$ give one marginal-contribution observation for every factor. One path contains $n+1$ model states. Since the baseline and fully perturbed endpoints are shared by all paths, $M$ sampled paths require at most
\begin{equation}
  \min(2^n,\,2+M(n-1))
\end{equation}
distinct model configurations before accounting for overlap among intermediate states. Model responses should be cached when the coalition response is deterministic, because different sampled paths may pass through the same subset. Beyond efficiency, sharing the permutation makes the estimator preserve the closure of the attribution exactly at every finite sample size.

\begin{lemma}[Per-sample structural guarantees]
\label{lem:exact}
Let $\widehat{\Delta T}_1^{(M)},\dots,\widehat{\Delta T}_n^{(M)}$ be computed
from~\eqref{eq:estimator} using a common set of sampled permutations
$\pi_1,\dots,\pi_M$. Then, for every $M\ge 1$ and every realisation of the
sample:
\begin{enumerate}
  \item[(i)] \emph{Completeness / efficiency, exactly:}
  $\displaystyle\sum_{i=1}^{n}\widehat{\Delta T}_i^{(M)} = v(\mathcal{N})-v(\varnothing)=v(\mathcal{N})$.
  \item[(ii)] \emph{Null factor, exactly:} if $v(S\cup\{i\})=v(S)$ for all
  $S\subseteq \mathcal{N}\setminus\{i\}$, then $\widehat{\Delta T}_i^{(M)}=0$.
  \item[(iii)] \emph{Linearity, exactly:} for scalars $a,b$ and responses $v,w$,
  if the same sampled permutations are used for $av+bw$, $v$, and $w$, then
  \[
  \widehat{\Delta T}_i^{(M)}(av+bw)
  =
  a\,\widehat{\Delta T}_i^{(M)}(v)
  +
  b\,\widehat{\Delta T}_i^{(M)}(w).
  \]
\end{enumerate}
\end{lemma}

\begin{proof}
For any single permutation $\pi$ the marginals~\eqref{eq:marginal} telescope along the chain it induces: $\sum_{i=1}^{n}\Delta_i(\pi) = \sum_{k=1}^{n}\bigl(v(S_k)-v(S_{k-1})\bigr) = v(\mathcal{N})-v(\emptyset)$. Averaging over $\pi_1,\dots,\pi_M$ preserves this identity term by term, giving~(i). For~(ii), a null factor has $\Delta_i(\pi)=0$ for every $\pi$, hence $\widehat{\Delta T}_i^{(M)}=0$ before any averaging. Statement~(iii) is immediate because~\eqref{eq:estimator} is a fixed linear functional of the response once the common permutations are drawn.
\end{proof}

Completeness/efficiency, no residual/purity, and the null-factor property are exact. Linearity is exact when common sampled paths are used. Endpoint-reversal antisymmetry is exact only for a reversal-closed sampling design, meaning that every sampled ordering $\pi=(i_1,\ldots,i_n)$ is included with the same weight as its reverse $\pi^R=(i_n,\ldots,i_1)$. Factor-exchange symmetry is generally exact only in expectation unless the sampling design is label-symmetric. Path independence, or uniqueness, is also asymptotic rather than finite-sample: the estimator depends on the sampled paths at finite $M$, although its expectation is the unique Shapley attribution.

For deterministic scalar responses with bounded marginal contributions and independent permutation draws, the number of sampled permutations needed to control all factor-level Monte Carlo errors grows only logarithmically in the number of factors.

\begin{proposition}[Sample complexity]
\label{prop:complexity}
Suppose the marginal contributions have bounded range
\begin{align}
  r = \max_i\left(\max_{\pi}\Delta_i(\pi) - \min_{\pi}\Delta_i(\pi)\right).
\end{align}

Fix $\varepsilon>0$ and $\delta\in(0,1)$. If
\begin{equation}
  M \;\ge\; \frac{r^{2}}{2\varepsilon^{2}}\,\log\frac{2n}{\delta},
  \label{eq:hoeffding}
\end{equation}
then $\max_i\bigl|\widehat{\Delta T}_i^{(M)}-\Delta T_i(v)\bigr|\le\varepsilon$ with
probability at least $1-\delta$.
\end{proposition}

\begin{proof}
For fixed $i$, Hoeffding's inequality applied to the bounded i.i.d.\ terms $\Delta_i(\pi_m)$ gives 
\begin{align}
  \mathbb{P}(|\widehat{\Delta T}_i^{(M)}-\Delta T_i(v)|\ge\varepsilon) \le 2\exp(-2M\varepsilon^{2}/r^{2}).
\end{align}
A union bound over the $n$ factors, set equal to $\delta$, yields~\eqref{eq:hoeffding}.
\end{proof}

Combining Proposition~\ref{prop:complexity} with the per-permutation cost, all $n$ factor contributions can be controlled to accuracy $\varepsilon$ uniformly with high probability using worst-case configuration complexity
\begin{equation}
  O\!\left(n r^2 \varepsilon^{-2}\log(n/\delta)\right),
\end{equation}
against $2^n$ configurations for exact recovery of the full response table. The number of sampled permutations $M$ is logarithmic in $n$; the total simulation cost also includes the $O(n)$ model states along each sampled path. Note that in practice, one generally does not know the marginal range $r$ in advance. If only a response range $R=\max_S v(S)-\min_S v(S)$ is known, then each marginal contribution lies in $[\min_S v(S)-\max_S v(S),\,\max_S v(S)-\min_S v(S)]$, giving the conservative bound $r\leq 2R$, not $r\leq R$. In applications, one can also monitor the empirical variance of the sampled marginal contributions.

A further useful technique from the Shapley-sampling literature is reversal-paired, or \emph{antithetic}, sampling, introduced by \cite{lomeli2019antithetic} and applied to Shapley values by \cite{mitchell2022sampling}. For a permutation $\pi=(i_1,\ldots,i_n)$, write $\pi^R=(i_n,\ldots,i_1)$ for its reversal and define
\begin{equation}
  \bar{\Delta}_i(\pi)
  =
  \frac{1}{2}\left\{\Delta_i(\pi)+\Delta_i(\pi^R)\right\}.
  \label{eq:paired_marginal}
\end{equation}
Drawing $K$ independent uniform permutations $\pi_1,\ldots,\pi_K$ gives the paired estimator
\begin{equation}
  \widehat{\Delta T}_{i,\mathrm{pair}}^{(K)}
  =
  \frac{1}{K}\sum_{k=1}^{K}\bar{\Delta}_i(\pi_k).
  \label{eq:paired_estimator}
\end{equation}
This estimator is unbiased because both $\pi$ and $\pi^R$ are uniformly distributed when $\pi$ is uniform. It uses $2K$ path evaluations. Let
\begin{equation}
  \gamma_i
  =
  \operatorname{Cov}_{\pi}\!\left(\Delta_i(\pi),\Delta_i(\pi^R)\right).
\end{equation}
Since $\Delta_i(\pi)$ and $\Delta_i(\pi^R)$ have the same variance $\sigma_i^2$,
\begin{equation}
  \operatorname{Var}\!\left(\widehat{\Delta T}_{i,\mathrm{pair}}^{(K)}\right)
  =
  \frac{\sigma_i^2+\gamma_i}{2K}.
  \label{eq:paired_variance}
\end{equation}
The independent estimator with the same number $2K$ of path evaluations has variance $\sigma_i^2/(2K)$. Thus reversal pairing reduces variance for factor $i$ exactly when $\gamma_i<0$; if the covariance is positive it can increase variance. The gain is therefore not automatic, but it is a standard antithetic-sampling effect. Appendix~\ref{app:example} gives a small illustration of this variance reduction, and Figure~\ref{fig:fair_sampling} shows its effect in the FaIR experiment. 

\section{FaIR illustration with a threshold-response metric}
\label{sec:fair_illustration}

We illustrate the sampling method with the FaIR reduced-complexity climate model \citep{leach2021fair} driven by RCMIP input data \citep{nicholls2020rcmip}. The experiment uses the SSP5--8.5 scenario and $n=14$ binary forcing factors so that the exact Shapley values can still be computed ($2^{14}=16384$ model evaluations) and compared to the sampling estimates. Three factors, $\mathrm{CO}_2$, $\mathrm{CH}_4$, and $\mathrm{N}_2\mathrm{O}$, are prescribed as FaIR concentration-driven greenhouse-gas species. The remaining eleven factors are prescribed effective-radiative-forcing components: aerosol--radiation interactions, aerosol--cloud interactions, tropospheric ozone, stratospheric ozone, land-use albedo, light-absorbing particles on snow and ice, stratospheric water vapour from methane oxidation, contrails, other well-mixed greenhouse gases, solar forcing, and volcanic forcing. A concentration factor is absent when it is held at its 1750 value; a forcing factor is absent when its anomaly from 1750 is set to zero. The notebook used to reproduce the FaIR experiment and all figures in this section is available with the replication code \citep{climate_interactions_code}.

A natural and climate-relevant nonlinear response is the cumulative threshold-exceedance or overshoot-degree-year quantity commonly used in scenario assessment \citep{kikstra2022ar6}. For coalition $S$, let $T_S(t)$ be the FaIR surface-temperature trajectory and let $T_\emptyset(t)$ be the trajectory for the empty-factor baseline. We define
\begin{equation}
  F(S)
  =
  \int_{1750}^{2100}
  \max\!\left(T_S(t)-T_\emptyset(t)-2^\circ\mathrm{C},\,0\right)\,dt,
  \label{eq:threshold_response}
\end{equation}
measured in degree-years, calculated under trapezoidal integration on annual time steps. The baseline-normalised game is again $v(S)=F(S)-F(\emptyset)$; in this run $F(\emptyset)=0$.

The full factorial table contains $2^{14}=16384$ FaIR coalitions. Exact evaluation gives $v(\mathcal{N})=141.30$ degree-years above the 2$^\circ$C threshold. The sum of the singleton responses is $108.85$ degree-years, so the non-additive residual is $32.45$ degree-years, or $23.0\%$ of the full response. The exact Shapley attributions are shown in Figure~\ref{fig:fair_shapley}. They sum to the full threshold-exceedance response, as required by efficiency. 

The same figure also shows the convergence behaviour of the two largest Shapley values: $\mathrm{CO}_2$ and $\mathrm{CH}_4$, with exact Shapley values of $126.81$ and $17.45$ degree-years, respectively. The sampled running estimates are accompanied by pointwise normal Monte Carlo intervals computed from the observed marginal-contribution variances, and show that the estimates quickly converge to the exact values with increasing sample size.

\begin{figure}[t]
  \centering
  \begin{subfigure}[t]{0.49\linewidth}
    \centering
    \includegraphics[width=\linewidth]{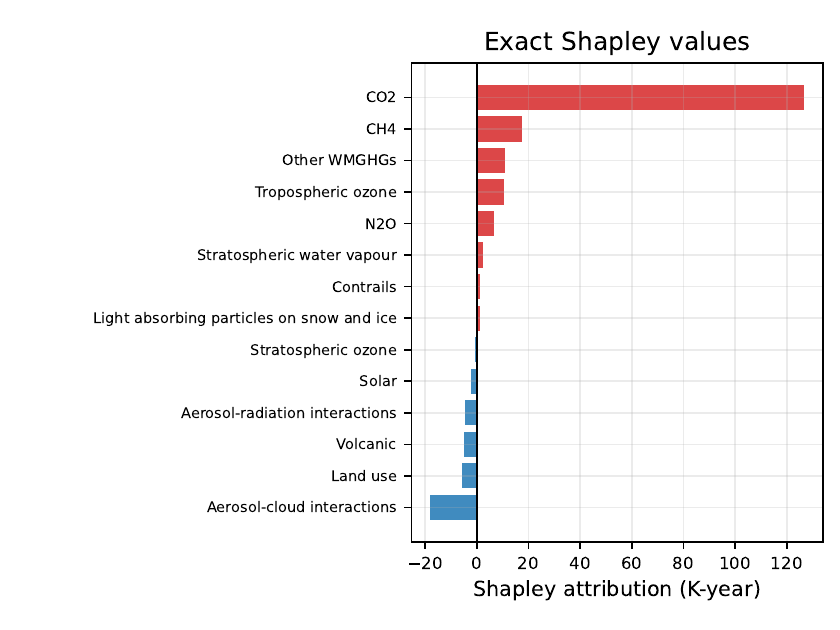}
    \caption{Exact 14-factor Shapley attributions calculated from the $2^{14}=16384$ FaIR model evaluations.}
    \label{fig:fair_shapley_values}
  \end{subfigure}
  \hfill
  \begin{subfigure}[t]{0.49\linewidth}
    \centering
    \includegraphics[width=\linewidth]{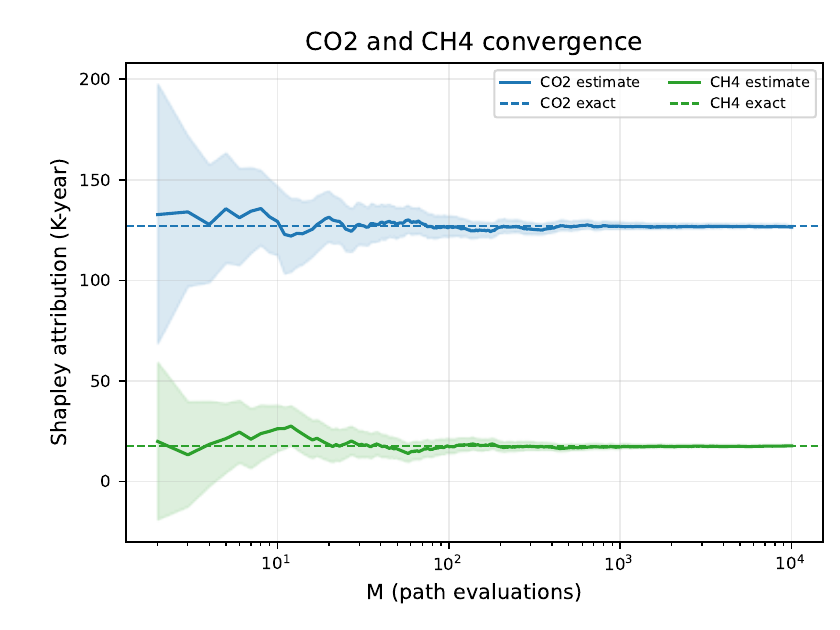}
    \caption{Running sampled estimates for $\mathrm{CO}_2$ and $\mathrm{CH}_4$ with pointwise 95\% Monte Carlo intervals. Dashed lines are the exact full-table values.}
    \label{fig:fair_ghg_ci}
  \end{subfigure}
  \caption{Exact and sampled Shapley attributions for the FaIR threshold-response experiment. Positive values increase cumulative degree-years above 2$^\circ$C by 2100; negative values reduce them.}
  \label{fig:fair_shapley}
\end{figure}

Figure~\ref{fig:fair_sampling} shows the corresponding sampling behaviour and the computational tradeoff. A model evaluation is one distinct FaIR coalition run, after caching repeated coalitions visited by different sampled paths. The error metric is the maximum absolute factor error, so it is a stringent absolute-error diagnostic but should not be read as guaranteeing small relative error for factors with Shapley values near zero. Across 200 independent sampling replications at $M=25$, the median estimate used 292 model evaluations, saved $98.2\%$ of the full-table runs, and had median maximum factor error $4.76$ degree-years. This is adequate for identifying the dominant contributions, but not for precise attribution of the smaller terms: aggregating over factor--trial observations, factors with exact $|\phi_i|<5$ degree-years had median relative absolute errors of about $14\%$, compared with about $10\%$ for factors with $|\phi_i|\geq 10$ degree-years. Reversal-paired sampling gave a better cost--accuracy point: at $M=50$ it used a median of 546 model evaluations, saved $96.7\%$ of the full table, and had median maximum error $0.52$ degree-years; the corresponding median relative errors were about $1.6$--$2.0\%$ for factors with $|\phi_i|<5$ and $1.1\%$ for factors with $|\phi_i|\geq 10$. Increasing to $M=10000$ reduced the median error to $0.253$ degree-years for independent sampling and $0.0388$ degree-years for reversal-paired sampling, but used a median of 15956 cached evaluations, only $2.6\%$ fewer than the exact table.

\begin{figure}[t]
  \centering
  \includegraphics[width=\linewidth]{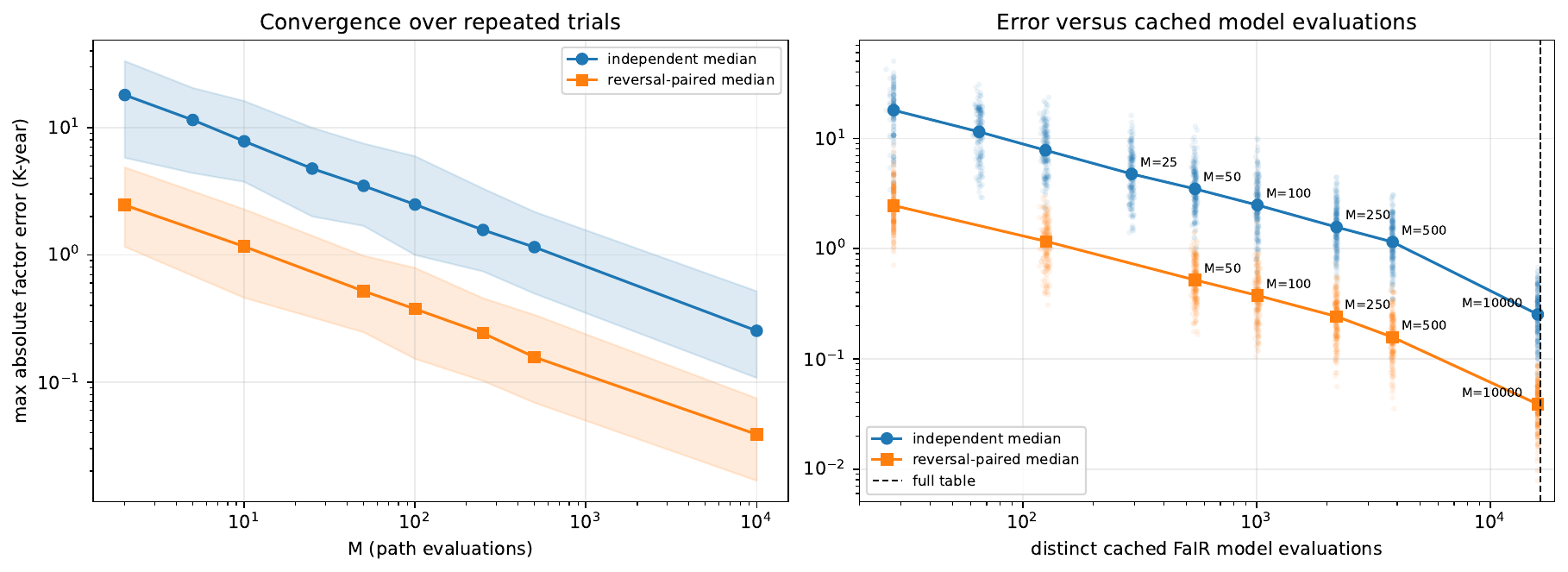}
  \caption{Sampling performance for the FaIR threshold-response experiment. Left: convergence to the exact Shapley values, shown as the median maximum absolute factor error over 200 independent sampling replications with 5th--95th percentile bands. Right: the same final errors plotted against actual computational cost, measured as distinct cached FaIR coalition evaluations. The dashed vertical line marks the 16384 model evaluations required by the exact full table.}
  \label{fig:fair_sampling}
\end{figure}

\section{Conclusion}
We have identified the factor interactions of \citet{stein1993factor} with the Boolean Möbius coefficients, or the Harsanyi dividends, of the baseline-normalised model response. Under this identification, the shared-interaction factorisation of \cite{lunt2021multi} is the dividend representation of the Shapley value, while their linear-sum factorisation is its permutation representation.

The two factorisations therefore coincide for any number of factors, proving the conjecture of \citet{lunt2021multi}. While this study was motivated by this conjecture, the main result is the identification of those two factor-separation attributions with a well-studied mathematical structure. This gives climate scientists a principled way to use established Shapley-value estimators and variance-reduction schemes \citep{castro2009polynomial,maleki2013bounding,castro2017improving,mitchell2022sampling}. 

The FaIR threshold-response illustration shows the practical distinction between exact interaction recovery and approximate factor attribution. Exact recovery required 16384 model evaluations, while sampled paths produced closed Shapley estimates at a tunable cost. Low sample sizes gave large savings but only coarse attribution: at $M=25$, the median maximum absolute error was about $4.8$ degree-years, which is small relative to the total response but material for several smaller factors. In relative terms, the less dominant factors had somewhat larger errors than the dominant factors. Reversal-paired sampling improved this tradeoff, reaching a median maximum error of about $0.5$ degree-years with 546 cached evaluations. The example is deliberately small enough to compute the exact benchmark, but it illustrates the scaling advantage that becomes essential when each coalition response is a costly climate-model simulation.

By framing the Stein--Alpert decomposition as a Möbius inversion, this work also connects factor separation to a wider literature on Möbius inversions and Shapley values for complex systems \citep{jansma2025mereological}, including extensions to vector-valued functions like wind velocity fields or chemical fluxes \citep{forre2025m}. 

While we have provided sampling error bounds for the permutation estimator, these are valid for deterministic models; we have not addressed the additional uncertainty from internal variability, finite integration, and parameter or structural uncertainty. Still, since the Shapley functional is linear, zero-mean noise in the model response only adds variance but does not bias the attribution.
By connecting the factor-separation literature to the Möbius inversion and Shapley value literature, we hope to encourage further work on uncertainty quantification/reduction, sampling methods, and interaction estimation in climate-model factor separation.

\appendix

\section{Three-factor sampling example}
\label{app:example}

Take $n=3$ factors with the baseline-normalised response in Table~\ref{tab:example}, a mildly superadditive game in which each pair and the triple carry one unit of irreducible interaction. The Möbius inversion~\eqref{eq:mobius} gives singleton dividends $\hat v(\{1\}),\hat v(\{2\}),\hat v(\{3\})=1,2,3$ and $\hat v(S)=1$ for every pair and for the triple. Both routes of Theorem~\ref{thm:main} return the same attribution, $\Delta T_1,\Delta T_2,\Delta T_3=\tfrac{7}{3},\tfrac{10}{3},\tfrac{13}{3}$, which sum to $10=v(\mathcal{N})-v(\varnothing)$ as required. Each of the six permutations telescopes exactly to $10$, illustrating Lemma~\ref{lem:exact}(i). The six permutation marginals for factor $1$ are $\{1,1,2,4,2,4\}$ with population variance $\sigma_1^{2}=14/9$; pairing each permutation with its reversal collapses these to the three reversal-pair means $\{\tfrac52,\tfrac52,2\}$, with variance $1/18$. Equivalently, $\operatorname{Cov}(\Delta_1(\pi),\Delta_1(\pi^R))=-13/9<0$, so~\eqref{eq:paired_variance} gives the variance reduction.

\begin{table}[t]
  \centering
  \begin{tabular}{lcccccccc}
    $S$    & $\varnothing$ & $\{1\}$ & $\{2\}$ & $\{3\}$
           & $\{1,2\}$ & $\{1,3\}$ & $\{2,3\}$ & $\{1,2,3\}$ \\
    $v(S)$ & $0$ & $1$ & $2$ & $3$ & $4$ & $5$ & $6$ & $10$ \\
  \end{tabular}
  \caption{A three-factor baseline-normalised response used to illustrate the
  sampled attribution. Both the dividend and permutation representations return
  $(\Delta T_1,\Delta T_2,\Delta T_3)=(\tfrac73,\tfrac{10}{3},\tfrac{13}{3})$,
  summing exactly to $v(\mathcal{N})-v(\varnothing)=10$.}
  \label{tab:example}
\end{table}

\bibliographystyle{plainnat}
\bibliography{refs}

@article{stein1993factor,
  title={Factor separation in numerical simulations},
  author={Stein, U and Alpert, P},
  journal={Journal of Atmospheric Sciences},
  volume={50},
  number={14},
  pages={2107--2115},
  year={1993}
}

@article{lunt2021multi,
  title={Multi-variate factorisation of numerical simulations},
  author={Lunt, Daniel J. and Chandan, Deepak and Haywood, Alan M. and Lunt, George M. and Rougier, Jonathan C. and Salzmann, Ulrich and Schmidt, Gavin A. and Valdes, Paul J.},
  journal={Geoscientific Model Development},
  volume={14},
  pages={4307--4317},
  year={2021},
  doi={10.5194/gmd-14-4307-2021}
}

@article{leach2021fair,
  title={{FaIR}v2.0.0: a generalized impulse response model for climate uncertainty and future scenario exploration},
  author={Leach, Nicholas J. and Jenkins, Stuart and Nicholls, Zebedee and Smith, Christopher J. and Lynch, John and Cain, Michelle and Walsh, Thomas and Wu, Buwen and Tsutsui, Junichi and Allen, Myles R.},
  journal={Geoscientific Model Development},
  volume={14},
  pages={3007--3036},
  year={2021},
  doi={10.5194/gmd-14-3007-2021}
}

@article{nicholls2020rcmip,
  title={Reduced Complexity Model Intercomparison Project Phase 1: introduction and evaluation of global-mean temperature response},
  author={Nicholls, Zebedee R. J. and Meinshausen, Malte and Lewis, Jared and Gieseke, Robert and Dommenget, Dietmar and Dorheim, Kalyn and Fan, Chien-So and Fuglestvedt, Jan S. and Gasser, Thomas and Gol{\"u}ke, Ulrike and Goodwin, Philip and Hartin, Corinne and Hope, Anthony P. and Kriegler, Elmar and Leach, Nicholas J. and Marchegiani, David and McBride, Laura A. and Quilcaille, Yann and Rogelj, Joeri and Salawitch, Ross J. and Samset, Bj{\o}rn H. and Sandstad, Marit and Shiklomanov, Alexey N. and Skeie, Ragnhild B. and Smith, Christopher J. and Smith, Steven and Tanaka, Katsumasa and Tsutsui, Junichi and Xie, Zhen},
  journal={Geoscientific Model Development},
  volume={13},
  pages={5175--5190},
  year={2020},
  doi={10.5194/gmd-13-5175-2020}
}

@article{kikstra2022ar6,
  title={The {IPCC} Sixth Assessment Report {WGIII} climate assessment of mitigation pathways: from emissions to global temperatures},
  author={Kikstra, Jarmo S. and Nicholls, Zebedee R. J. and Smith, Christopher J. and Lewis, Jared and Lamboll, Robin D. and Byers, Edward and Sandstad, Marit and Meinshausen, Malte and Gidden, Matthew J. and Rogelj, Joeri and Kriegler, Elmar and Peters, Glen P. and Fuglestvedt, Jan S. and Skeie, Ragnhild B. and Samset, Bj{\o}rn H. and Wienpahl, Laura and van Vuuren, Detlef P. and van der Wijst, Kaj-Ivar and Al Khourdajie, Alaa and Forster, Piers M. and Reisinger, Andy and Schaeffer, Roberto and Riahi, Keywan},
  journal={Geoscientific Model Development},
  volume={15},
  pages={9075--9109},
  year={2022},
  doi={10.5194/gmd-15-9075-2022}
}

@article{jansma2025mereological,
  title={Mereological approach to higher-order structure in complex systems: From macro to micro with {M}{\"o}bius},
  author={Jansma, Abel},
  journal={Physical Review Research},
  volume={7},
  number={2},
  pages={023016},
  year={2025},
  publisher={APS}
}

@article{forre2025m,
  title={M{\"o}bius transforms and {S}hapley values for vector-valued functions on weighted directed acyclic multigraphs},
  author={Forr{\'e}, Patrick and Jansma, Abel},
  journal={arXiv preprint arXiv:2510.05786},
  year={2025}
}

@incollection{harsanyi1959bargaining,
  title={A Bargaining Model for the Cooperative n-Person Game},
  author={Harsanyi, John C.},
  booktitle={Contributions to the Theory of Games IV},
  editor={Tucker, A. W. and Luce, R. D.},
  pages={325--355},
  year={1959},
  publisher={Princeton University Press}
}

@incollection{shapley1953value,
  title={A Value for n-Person Games},
  author={Shapley, Lloyd S.},
  booktitle={Contributions to the Theory of Games II},
  editor={Kuhn, H. W. and Tucker, A. W.},
  pages={307--317},
  year={1953},
  publisher={Princeton University Press}
}

@article{cleveland2020factor,
  title={Factor Effects in Numerical Simulations},
  author={Cleveland, Judah L. and Smith, Jeffrey A. and Collins, James P.},
  journal={Journal of the Atmospheric Sciences},
  volume={77},
  number={7},
  pages={2439--2451},
  year={2020},
  doi={10.1175/JAS-D-19-0263.1}
}

@article{albrecht2002shapley,
  title={A {S}hapley Decomposition of Carbon Emissions Without Residuals},
  author={Albrecht, Johan and Fran{\c{c}}ois, Delphine and Schoors, Koen},
  journal={Energy Policy},
  volume={30},
  number={9},
  pages={727--736},
  year={2002},
  doi={10.1016/S0301-4215(01)00131-8}
}

@article{alkhourdajie2024climate,
  title={Climate Ambition, Background Scenario or the Model? Attribution of the Variance of Energy-Related Indicators in Global Scenarios},
  author={Al Khourdajie, Alaa and Skea, Jim and Green, Richard},
  journal={Energy and Climate Change},
  volume={5},
  pages={100126},
  year={2024},
  doi={10.1016/j.egycc.2024.100126}
}

@article{owen2014sobol,
  title={{Sobol'} Indices and {S}hapley Value},
  author={Owen, Art B.},
  journal={SIAM/ASA Journal on Uncertainty Quantification},
  volume={2},
  pages={245--251},
  year={2014},
  doi={10.1137/130936233}
}

@article{song2016shapley,
  title={Shapley Effects for Global Sensitivity Analysis: Theory and Computation},
  author={Song, Eunhye and Nelson, Barry L. and Staum, Jeremy},
  journal={SIAM/ASA Journal on Uncertainty Quantification},
  volume={4},
  number={1},
  pages={1060--1083},
  year={2016},
  doi={10.1137/15M1048070}
}

@article{castro2009polynomial,
  title={Polynomial calculation of the {S}hapley value based on sampling},
  author={Castro, Javier and G{\'o}mez, Daniel and Tejada, Juan},
  journal={Computers \& Operations Research},
  volume={36}, number={5}, pages={1726--1730}, year={2009}}

@article{maleki2013bounding,
  title={Bounding the estimation error of sampling-based {S}hapley value approximation},
  author={Maleki, Sasan and Tran-Thanh, Long and Hines, Greg and Rahwan, Talal and Rogers, Alex},
  journal={arXiv preprint arXiv:1306.4265}, year={2013}}

@article{castro2017improving,
  title={Improving polynomial estimation of the {S}hapley value by stratified random sampling with optimum allocation},
  author={Castro, Javier and G{\'o}mez, Daniel and Molina, Elisenda and Tejada, Juan},
  journal={Computers \& Operations Research},
  volume={82}, pages={180--188}, year={2017}}

@article{mitchell2022sampling,
  title={Sampling permutations for {S}hapley value estimation},
  author={Mitchell, Rory and Cooper, Joshua and Frank, Eibe and Holmes, Geoffrey},
  journal={Journal of Machine Learning Research},
  volume={23}, number={43}, pages={1--46}, year={2022}}

@article{rota1964foundations,
  title={On the foundations of combinatorial theory I. Theory of {M}{\"o}bius Functions},
  author={Rota, Gian-Carlo},
  journal={Zeitschrift f{\"u}r Wahrscheinlichkeitstheorie und Verwandte Gebiete},
  volume={2},
  number={4},
  pages={340--368},
  year={1964},
  publisher={Springer Science and Business Media LLC}
}

@article{lundberg2017unified,
  title={A unified approach to interpreting model predictions},
  author={Lundberg, Scott M and Lee, Su-In},
  journal={Advances in neural information processing systems},
  volume={30},
  year={2017}
}

@book{mann1960values,
  title={Values of large games, IV: Evaluating the electoral college by Montecarlo techniques},
  author={Mann, Irwin and Shapley, Lloyd S},
  year={1960},
  publisher={Rand Corporation}
}

@misc{climate_interactions_code,
  title={Replication notebook for climate-model Shapley experiments},
  author={Abel Jansma},
  year={2026},
  howpublished={\href{https://github.com/AJnsm/climate_shapley_values}{\url{https://github.com/AJnsm/climate_shapley_values}}}
}

@article{lomeli2019antithetic,
  title={Antithetic and Monte Carlo kernel estimators for partial rankings},
  author={Lomeli, Maria and Rowland, Mark and Gretton, Arthur and Ghahramani, Zoubin},
  journal={Statistics and Computing},
  volume={29},
  number={5},
  pages={1127--1147},
  year={2019},
  publisher={Springer}
}

\end{document}